\documentclass{amsart}
\usepackage{setspace}
\usepackage{latexsym, amsmath}
\usepackage{amsfonts}
\usepackage{amssymb}
\usepackage{amsthm}
\usepackage[all]{xy}
\usepackage[colorlinks=true,linkcolor=black,citecolor=black]{hyperref}

\theoremstyle{plain}
\newtheorem{thm}{Theorem}
\newtheorem{theorem}{Theorem}[section]
\newtheorem{proposition}[theorem]{Proposition}

\newtheorem{lemma}[theorem]{Lemma}

\newtheorem*{conj*}{Conjecture}

\theoremstyle{definition}
\newtheorem{example}[theorem]{Example}

\newtheorem{remark}[theorem]{Remark}

\numberwithin{equation}{section}

\newcommand{\Z}{{\mathbb Z}}
\newcommand{\Q}{{\mathbb Q}}
\newcommand{\C}{{\mathbb C}} 
\newcommand{\R}{{\mathbb R}}

\newcommand{\comment}[1]{}
\def\im{\operatorname{Im}} \def\ker{\operatorname{Ker}} 
  
\def\Symb{\operatorname{Symb}}
\def\PP{\mathbb{P}} \def\Q{\mathbb{Q}}\def\R{\mathbb{R}}\def\Z{\mathbb{Z}}\def\C{\mathbb{C}}
\def\le{\leqslant} \def\ge{\geqslant}
\def\tr{\operatorname{Tr}}
\def\ind{\operatorname{Ind}}\def\cor{\operatorname{cor}}\def\res{\operatorname{res}}
\def\ch{\operatorname{char}}\def\rk{\operatorname{rank}}
\def\SL{\mathrm{SL}} \def\PSL{\mathrm{PSL}}\def\GL{\mathrm{GL}}
\def\I{\mathcal{I}}  
 
  \def\M{\mathcal{M}}\def\RR{\mathcal{R}}
\def\H{\mathcal{H}}

\def\e{\varepsilon} \def\DD{\Delta} \def\G{\Gamma}
\def\SS{\Sigma}
\def\dd{\delta} \def\ss{\sigma}
  
\def\a{\alpha}\def\b{\beta}\def\g{\gamma}
\def\+{\,+\,}  \def\={\;=\;} 
\def\sm#1#2#3#4{\left(\begin{smallmatrix}#1&#2 \\ #3 & #4 \end{smallmatrix}\right)}
\def\be{\begin{equation}}  \def\ee{\end{equation}}

\def\ov#1{\overline{#1}}
\newcommand{\wT}{\widetilde{T}}
\renewcommand{\wr}{\widehat{\rho}}
\newcommand{\wW}{\widehat{W}}
\def\new{\mathrm{new}}
\def\bsh{\backslash}

\def\vp{\varphi}
\def\rar{\rightarrow}\def\FF{\mathbb{F}}
\def\la{\langle} \def\ra{\rangle}

\def\den{\mathrm{Den}\;}
\def\lrar{\longrightarrow}
\title[A generalization of Ramanujan's congruence]
{A generalization of Ramanujan's congruence\\ to modular forms of prime level}
\author{Radu Gaba and Alexandru A. Popa}
\thanks{The authors were supported by CNCSIS grant TE-2014-4-2077.}

\address{Institute of Mathematics ``Simion Stoilow" of the Romanian Academy,
P.O. Box 1-764, RO-014700 Bucharest, Romania}
\address{E-mail: radu.gaba@imar.ro, alexandru.popa@imar.ro }

\begin{document}

\begin{abstract} We prove congruences between cuspidal newforms and
Eisenstein series of prime level, which generalize Ramanujan's congruence. 
Such congruences were recently found by Billerey and Menares, 
and we refine them by specifying the Atkin-Lehner eigenvalue of the newform 
involved. We show that similar refinements hold for the level raising congruences 
between cuspidal newforms of different levels, due to Ribet and Diamond. 
The proof relies on studying the new subspace and the Eisenstein subspace 
of the space of period polynomials for the congruence subgroup~$\G_0(N)$, 
and on a version of Ihara's lemma. 
\end{abstract}
\maketitle

\section{Introduction}

Let $E_k$ be the Eisenstein series of even weight $k\ge 4$ for the full
modular group, normalized so that its Fourier expansion is
 \[ E_k(z)=-\frac{B_k}{2k}+\sum_{n\ge 1} \sigma_{k-1}(n) q^n, \]
where  $B_k$ is the Bernoulli number, $\sigma_{a}(n)=\sum_{d|n} d^a$, 
and $q=e^{2\pi i z}$. Let $\I$ be a prime ideal dividing 
the numerator of $\frac{B_k}{2k}$, in the number field generated by 
the eigenvalues of Hecke eigenforms of weight $k$.  
Then there exists such a cuspidal Hecke eigenform $f$ such that 
 \be \label{e0} f\equiv E_k \pmod{\I};\ee
for $k=12$ this is the well-known Ramanujan congruence modulo 691, 
while for higher weights it was proved in \cite{H,DG}. This and later 
congruences mean that the difference between the 
Fourier coefficients at the cusp $\infty$ of the two sides belong to~$\I$
(after clearing the denominator of the constant term), and we always normalize
Hecke eigenforms to have the coefficient of~$q$ equal to 1. 

This congruence was recently generalized to newforms $f$ of prime level by 
Billerey-Menares~\cite{BM} and by Dummigan-Fretwell~\cite{DF}, for 
Fourier coefficients of index coprime to the level. In this paper 
we refine these results, by determining also the Atkin-Lehner eigenvalue 
of the newform involved, thus obtaining congruences for all coefficients.
A similar congruence between cuspidal newforms of different levels is the 
``level raising theorem'' of Ribet~\cite{Ri} and Diamond~\cite{Di}, and 
we show that it admits a similar refinement. Before stating the results, 
we introduce the common setting and explain the heuristic behind them. 

Let $S_k(N)$, $M_k(N)$ be the space of cusp forms, respectively modular forms
of even weight $k\ge 2$ and trivial Nebentypus for the congruence subgroup $\G_0(N)$. 
We let $N=Mp$ with $p\nmid M$ a prime, and consider a modular form  
$g\in M_k(M)$. Fixing $\e\in\{\pm 1\}$,  
we define 
\be\label{1.2}
g^{(\e)}_p:=g|(1+\e W_p),\quad \text{ namely } 
g^{(\e)}_p(z):=g(z)+\e p^{k/2}g(pz),
\ee
which is a form of level $N$ with eigenvalue $\e$ under the Atkin-Lehner 
involution $W_p$.  Since~$g$ has level $M$, it follows that
$\tr^N_M (g|W_p) = p^{-k/2+1} g|T_p$,
where $\tr^N_M:M_k(N)\rar M_k(M)$ is the trace map (the definition is recalled in Section~\ref{s2.3}). 
Assuming further that $g$ is an eigenform of the Hecke operator $T_p$ with 
eigenvalue $\lambda_p$ we have
\be\label{2.12}
     \tr^N_M g^{(\e)}_p =\big(1+p+\e p^{-k/2+1} \lambda_p\big) \cdot  g\;, 
\ee
using that $[\G_0(M):\G_0(N)]=1+p$. The cuspforms $f\in S_k(N)$ which 
are new at $p$ can be characterized by the condition $\tr^N_M f=\tr^N_M f|W_N =0$, as recalled in
Section~\ref{s2.3}, where we also recall the definitions. We conclude from~\eqref{2.12} that if $g$ is a newform of level $M$,
then $g^{(\e)}_p$ is new at~$p$ when reduced modulo prime ideals dividing the term in parentheses,
and heuristically we expect that it is congruent modulo such ideals to a Hecke eigenform in 
$S_k(N)$, which is new at~$p$ and has eigenvalue~$\e$ under~$W_p$. 
The next two theorems confirm this heuristic.  

When $M=1$ and $g=E_k$, denote by $E_{k,p}^{(\e)}$ the form~$g_p^{(\e)}$
in~\eqref{1.2}.  The factor in~\eqref{2.12} splits: 
\[ 1+p+\e p^{-k/2+1}\ss_{k-1}(p) =(p^{-k/2+1}+\e)(p^{k/2}+\e), \]
and since $W_p$ interchanges 
the cusps $\infty$ and 0 of $\G_0(p)$, the constant terms of 
$E_{k,p}^{(\e)}$ at both cusps are $\frac{B_k}{2k}(p^{k/2}+\e)$ 
(up to a sign and powers of $p$). Therefore candidates for congruence primes 
between $E_{k,p}^{(\e)}$ and newforms are prime ideals which 
divide its constant terms and the product above. 

We denote by  $S_k^{(\e)}(N)$ the subspace of $S_k(N)$ consisting of 
eigenforms for~$W_p$ with eigenvalue~$\e$. We write $a|q$, $a\nmid q$ if the 
integer $a$ divides, respectively does not divide, the numerator of the rational
number $q$. 

\begin{thm}\label{T1} Let $k\ge 4$ be even, $p$ a prime, and~$\e\in\{\pm 1\}$. 
Let $\ell$ be a prime with $\ell>k+1$, and 
assume that the following conditions are satisfied: 
\be \label{1.6}
\ell\big|(p^{k/2}+\e)(p^{k/2-1}+\e) \text{\ \  and \ \
   $\ell\big| \frac{B_k}{k}(p^{k/2}+\e).$}
\ee 
If $\ell\nmid (p^{k/2}+\e)$, assume also that $\ell\nmid B_nB_{k-n}(p^{n-1}-1)$ for some even $n$, $0<n<k$.
Then there exists a prime ideal  $\I$ of residue characteristic 
$\ell$, in the ring of integers
generated by Hecke eigenvalues of newforms in $S_k^{(\e)}(p)$, and a newform 
$f\in S_k^{(\e)}(p)$ such that
\be \label{e1} 
f \equiv E_{k,p}^{(\e)} \pmod{\I}.
\ee
\end{thm}
We prove the theorem in Section~\ref{s3}, together with the next theorem.
The theorem refines~\cite[Thm. 1]{BM}, where it is shown 
that a congruence as in~\eqref{e1} holds for coefficients coprime to $p$ if and 
only if $\ell|(p^k-1)(p^{k-2}-1)$ and $\ell$ divides the numerator of 
$\frac{B_k}{k}(p^{k}-1)$. 
The additional condition we impose if $\ell\nmid (p^{k/2}+\e)$ is an artifact of our method, 
but it is automatically satisfied for all $k\le 6\cdot 10^4$ (see Remark~\ref{r2}).  We 
also checked numerically that the theorem holds for $\ell=k\pm 1$ in numerous cases (see 
Example~\ref{ex1}), but our method does not apply for these values of $\ell$.

When the form $g$ is a cuspidal newform, we obtain instead the following 
refinement of Diamond's level raising theorem~\cite[Thm. 1]{Di}. We give this theorem
as an illustration of our method, as it requires little extra work. The statement is not the 
sharpest possible, since assumption~\eqref{1.5} below could probably be removed.\newpage

\begin{thm}\label{T3}Let $k\ge 4$ be even, 
let $N=Mp$ with $p$ prime, $p\nmid M$,  and $\e\in\{\pm 1\}$.
Let $g\in S_k(M)$ be a newform with eigenvalue $\lambda_p$ under $T_p$. 
Assume there is  a prime $\ell>k+1$, $\ell\nmid N$, and a prime ideal $\I$ above it
in the field generated by the eigenvalues of all Hecke eigenforms in $S_k^{(\e)}(N)$
such that 
\[ \lambda_p\equiv -\e p^{k/2-1}(p+1)\pmod \I . \]
Assume also that either 
\be\label{1.5}  \ell\nmid (p^{k/2-1}+\e) \cdot \den P^+(g),\quad
\text{ or }\  k\ge 6 \text{ and } \ell\nmid (p^{k/2-2}+\e) \cdot \den P^-(g) .
\ee 
Then there exists a Hecke eigenform
$f\in S_k^{(\e)}(N)$ which is new at~$p$ such that 
$$f\equiv g^{(\e)}_p \pmod \I.$$
\end{thm}
In condition~\eqref{1.5}, $P^+(g)$, respectively $P^-(g)$  is the even, 
respectively the odd period polynomial of $g$, 
normalized so that its principal part is $1$ modulo $X^2K_g[X]$, respectively $X$ 
modulo $X^3K_g[X]$, with $K_g$ the field generated by the Hecke eigenvalues of $g$.  
We denote by $\den P^{\pm}(g)$ the least common multiple of the norms of the 
denominators of the coefficients of $P^\pm(g)$ (which belong to $K_g$). 
These denominators tend to have few factors of residue characteristic 
$\ell> k$ (none if $f$ is of full level and weight $k$ such that $S_k(1)$ is one 
dimensional, cf. the tables in~\cite{Ma,Du}), so this condition is typically 
verified for a given~$g$ (see Example~\ref{ex3}). 
The theorem should hold for weight $k=2$ as well, without assumption~\eqref{1.5} but 
assuming that $\ell\nmid \vp(M)$,
which would refine the original result of Ribet~\cite{Ri}.

The proof of both theorems relies on the theory of period polynomials for 
congruence subgroups developed by Pa\c sol and the second author in~\cite{PP},
and the results we obtain along the way are of independent interest. 
In Section~\ref{S2} we define the new subspace of the space~$W_{w}(N)$
of period polynomials of degree~$w$  for~$\G_0(N)$, where $N\ge 1$ is arbitrary. By studying 
the action of the Atkin-Lehner involution on period polynomials,
we determine explicitly a basis for the ``Eisenstein subspace'' of~$W_{w}(N)$ 
consisting of Atkin-Lehner eigenvectors, when $N$ is square-free. 
We also need the larger space $\wW_{w}(N)$ of extended 
period polynomials introduced in~\cite{PP}, and we use it to prove an Eichler-Shimura 
isomorphism between the new subspaces of~$\wW_{k-2}(N)$ and 
of~$M_k(N)$
(see Propositions~\ref{P8} and~\ref{C2}). Compared to 
the better known theory of modular symbols, there are two new features: 
the new subspace is defined using a trace map from higher to lower levels, as in
Serre's characterization of newforms; and to associate even period polynomials
in $W_{k-2}(N)$ to Eisenstein series when~$N$ is square-free, we require 
the larger space~$\wW_{k-2}(N)$.

The upshot of the theory in Section~\ref{S2} is that the 
Eisenstein series in Theorem~\ref{T1} has an extended period polynomial 
$\wr\big(E_{k,p}^{(\e)}\big)$ whose even part belongs to the new subspace 
$W_w(p)_{/\FF_\ell}^\new$,
when reduced modulo primes $\ell|p^{k/2}+\e$, and similarly for the 
odd part modulo primes $\ell$ satisfying~\eqref{1.6} and  $\ell\nmid p^{k/2}+\e$. 
The even part is clearly nonzero modulo $\ell$, but for the 
nonvanishing of the odd part we require the extra condition in Theorem~\ref{T1}. 
Similarly, the polynomials $P^\pm\big(g_p^{(\e)}\big)$ in Theorem~\ref{T3} 
are new mod~$\I$ because of the congruence satisfied by $\lambda_p$, and at least one 
is nonzero mod~$\I$ by assumption~\eqref{1.5}. 
Since $\wr^\pm\big(E_{k,p}^{(\e)}\big)$ and  $P^\pm\big(g_p^{(\e)}\big)$ are Hecke and Atkin-Lehner
eigenvectors, the previous theorems follow from the Deligne-Serre 
lifting lemma~\cite{DeSe} (see Section~\ref{s3} for the details), once we establish the surjectivity 
of a reduction map on newspaces. 

This is the other main result, and to state it, we let $R$ be a discrete 
valuation ring with residue field~$F$ of characteristic~$\ell$.
Let  $W_{w}(N)^{p-\new}_{/R}$
be the space of polynomials new at $p|N$, defined over $R$ (see~\eqref{new} 
for the definition). 
\begin{thm} \label{T2} 
Let $w\ge 0$ be even, and let $N=pM$ with $p$ prime, $p\nmid M$.
Assume that $\ell> w+3$, $\ell\nmid N$, and if $w=0$ assume also that $\ell\nmid \vp(M)$.   
Then the reduction map
\[ W_{w}(N)^{p-\new}_{/R} \rar W_{w}(N)^{p-\new}_{/F} \]
is surjective (when $\ell=w+3>3$, its image has codimension 1). 
\end{thm}
For the proof, given in Section~\ref{S4}, we use the isomorphism between 
$W_{w}(N)_{/R}$ and the compactly supported cohomology
$H_c^1(\G_0(N), V_{w}(R))$, with $V_w(R)$ the module of polynomials  
of degree at most~$w$ with coefficients in~$R$. Using Poincar\'e duality,  
we show in Proposition~\ref{P5} that the surjectivity reduces to a version 
of Ihara's lemma~\cite[Lemma 3.2]{Ih}. Ihara's lemma, or rather the ingredients
in its proof, was first used to prove the level raising 
congruence mentioned above by Ribet (for weight two) and Diamond (for higher weights). 
We follow the argument in~\cite{Di}, modified to 
take into account that we work with the whole cohomology rather than just 
the parabolic part. Our work is close in spirit to Harder's program of 
proving congruences by studying Eisenstein cohomology classes~\cite{Ha}. 

The surjectivity of the map in Theorem~\ref{T2}, or lack thereof for $\ell= w+3$, 
can be traced to the vanishing of the finite cohomology group 
$H^1(\SL_2(\FF_\ell), V_w(F))$ for $\ell>w$, and its nonvanishing for~$\ell=w+3$. 
This was shown in~\cite{KPS}, but for completeness we give the proof in 
Proposition~\ref{P4}. 

We end the introduction with two remarks and a conjecture related to 
Theorem~\ref{T1}. 

Theorem~\ref{T1} is true in weight two as well, when it refines a congruence 
due to Mazur~\cite[Prop. 5.12]{Mz}. In this case, the Eisenstein subspace of $M_2(p)$ is 
one dimensional spanned by an Eisenstein series having Atkin-Lehner eigenvalue $-1$, and 
Theorem~\ref{T1} predicts that the newform in Mazur's congruence 
can be taken to have the same Atkin-Lehner eigenvalue. However there are technical 
difficulties in applying Theorem~\ref{T2} for $w=0$ in this case; 
the refined congruence is anyway proved by  Yoo~\cite[Thm. 1.3 (i)]{Yo}, who studies  
the case of weight two and square-free level in great detail, using different methods.  

Assuming a conjecture of Maeda, the Hecke eigenforms of full level form a 
single Galois orbit, so all of them satisfy~\eqref{e0} 
modulo conjugate ideals. Similarly, a generalization of 
Maeda's conjecture due to Tsaknias \cite{Ts} states that the newforms in 
$S_k(p)$ form two Galois orbits for sufficiently large $k$, 
the forms in each orbit sharing the same Atkin-Lehner eigenvalue. 
This would imply that all newforms in $S_k(p)$ satisfy congruence~\eqref{e1}, 
when primes~$\ell$ as in Theorem~\ref{T1} exist. 

A conjecture generalizing Ramanujan's congruence to newforms of 
square-free levels was proposed by Billerey and Menares in~\cite{BM}, 
and we end by stating a conjectural refinement that includes a generalization 
of Theorem~\ref{T3} as well. Let $N=M N'$ be 
square-free, $k\ge 4$, and let $g\in M_k(M)$ be a newform of level $M$. 
This includes the case of Eisenstein series, when we necessarily 
have $M=1$ and $g=E_k$. Let
$D(N')$ denote the set of positive divisors of $N'$, and let $\e:D(N')\rar\{\pm 1\}$ 
be a  multiplicative function, which we view as a system of Atkin-Lehner eigenvalues
for modular forms in $M_k(N)$. Define 
\[g_{N'}^{(\e)}=g| \prod_{p|N'} (1+\e(p) W_p), \text{ namely } 
g_{N'}^{(\e)}(z)=\sum_{d|N'}\e(d)d^{k/2} g(dz),
\]
which is a eigenform for $W_p$ for prime $p|N'$ with eigenvalue $\e(p)$, and 
it is new at primes dividing~$M$. Note that when $M=1$ and $g=E_k$, 
the constant term of $g_{N}^{(\e)}$ at all cusps is 
(up to signs and powers of $p|N$):
\be\label{1.7} \frac{B_k}{2k}\prod_{p|N}(p^{k/2}+\e(p))\;, \ee 
since the group of Atkin-Lehner involutions acts transitively on the cusps for square-free~$N$. 
Let $S_k^{(\e)}(N)$ be the subspace of $S_k(N)$ 
consisting of eigenforms for $W_p$ with eigenvalue $\e(p)$, for all $p|N'$. 

\begin{conj*} Assume $N=MN'$ square-free and $k\ge 4$, and  let $g\in M_k(M)$ 
be a newform of level $M$, with Hecke eigenvalues $\lambda_p$ for $p|N'$. 
Let $\e:D(N')\rar \{\pm 1\}$ be a system of Atkin-Lehner eigenvalues as above, 
and assume that for all $p|N'$ we have 
$$\lambda_p \equiv -\e(p)p^{k/2-1}(p+1) \pmod \I \;,$$
for $\I$ a prime ideal of residue characteristic $\ell$ in the field 
generated by the Hecke eigenvalues of all newforms in $S_k^{(\e)}(N)$, such that 
$\ell>k-2$, $\ell\nmid 6N$. If $M=1$ and $g=E_k$ also assume 
that $\ell$ divides the numerator of~\eqref{1.7}. 
Then there exists a newform $f\in S_k^{(\e)}(N)$ such that
\[f\equiv g_{N'}^{(\e)} \pmod \I. \]
\end{conj*}
By~\eqref{2.12}, the conditions in the conjecture guarantee that $g_{N'}^{(\e)}$ 
is cuspidal and ``new'' when reduced modulo the ideal $\I$. The reductions mod~$\I$
make sense in the space $S_k(\G_0(N), \ov{\Z}_\ell)$ of arithmetic modular forms 
\`a la Katz~\cite{Ed}. Provided 
one had a definition of the ``new subspace'' $S_k(\G_0(N), \ov{\Z}_\ell)^\new$ 
involving trace maps, as for modular forms over $\C$, then the conjecture 
would follow immediately from the surjectivity of the reduction map 
\[S_k(\G_0(N), \ov{\Z}_\ell)^\new \rar S_k(\G_0(N), \ov{\FF}_\ell)^\new . 
\]
Part of the conjecture would also follow from a generalization of Theorem~\ref{T2}, stating that the 
reduction map $W_{w}(N)^{\new}_{/R} \rar W_{w}(N)^{\new}_{/F}$
is surjective for $N$ square-free. 

For cuspidal $g$, the existence of a newform $f$ as in the conjecture--with possibly non-trivial 
Nebentypus and minus the determination
of its Atkin-Lehner eigenvalues at primes $p|N'$--follows from a theorem of 
Diamond and Taylor on non-optimal levels for modular Galois 
representations~\cite{DT}.  That theorem was proved for arbitrary $M$ 
such that $(M,N')=1$ and $N'$ square-free, and we similarly 
expect the conjecture to hold in the cuspidal case under these assumptions. 
When $g$ is an Eisenstein newform of level $M>1$, a similar conjecture can be 
made, under extra assumptions due to the fact that the group of Atkin-Lehner involutions
no longer acts transitively on the cusps of~$\G_0(M)$.

For $k=2$ and $g$ an Eisenstein series, similar 
statements have been proved by Yoo~\cite{Yo} (who also determines the  
Atkin-Lehner eigenvalues of~$f$), and Martin~\cite{Mar}.

{\bf Acknowledgments.} We would like to thank Vicen\c tiu Pa\c sol 
for many useful conversations. We are grateful to Neil Dummigan and 
Dan Fretwell for bringing to our attention reference~\cite{BM}, after a first
version of this paper was released as a preprint. 
We thank a referee for a thorough reading of the paper and for pointing out an error in 
an earlier version of Lemma~\ref{L3}.

\section{Atkin-Lehner operators and newform theory for period polynomials}
\label{S2}

In~\S\ref{s2.1}, we briefly review the definition of period polynomials 
for $\G_0(N)$ and the action of Hecke operators on them. We study 
the action of Atkin-Lehner operators in more detail in~\S\ref{s2.2}, 
and in~\S\ref{s2.3} we define primitive (new) subspaces using trace maps. In~\S\ref{s2.40}
we introduce extended period polynomials, and in~\S\ref{s2.4} we use them 
 to determine an explicit basis consisting of Atkin-Lehner eigenvectors of the 
 ``Eisenstein subspace'' for square-free~$N$. 

\subsection{Period polynomials and Hecke operators} \label{s2.1}
We start  by recalling from~\cite{PP} basic facts about
period polynomials for the congruence subgroup $\G=\G_0(N)$ of
$\G_1=\SL_2(\Z)$. Let $S=\sm 0{-1}10$, $U=\sm 1{-1}10$ be generators of 
$\G_1$. We use the same notation for their images in $\PSL_2(\Z)$, which have
orders 2, 3, respectively. 

We fix a commutative base ring $R$ of characteristic
different from 2 and 3.
Let $V_w(R)$ be the space of polynomials of degree at most $w$ 
with coefficients in $R$, on which $\GL_2(\Z)$ acts on the right by 
\[ P|_{-w}\g(X)=P(\g X) (cX+d)^{w}, \text{  for $\g=\sm abcd\in \GL_2(\Z$).} \]
Let $V_{w}(N)_{/R}$ 
be the space of $|\G\bsh\G_1|$-tuples of polynomials,\footnote{Note that we have
two notations for $V_w(R)=V_w(1)_{/R}$, but for brevity we use the shorter notation.}
identified with maps $P:\G\bsh \G_1 \rightarrow V_w(R) $, on which $\G_1$ acts by 
$P|\g(A)=P(A\g^{-1})|_{-w} \g$ for a coset $A\in \G\bsh \G_1$, $\g \in\G_1$. 
Since $-1\in\G$ we assume $w\ge 0$ is even, so $-1$ acts trivially on $V_w(N)_{/R}$. 
The space of \emph{period polynomials} is defined by:
\[W_w(N)_{/R}=\{P\in V_w(N)_{/R} :  P|(1+S)=P|(1+U+U^2)=0\}. \]
The element $\dd=\sm {-1}001\in \GL_2(\Z)$ belongs to the normalizer of $\G$, 
so $W_w(N)_{/R}$ is preserved
by the involution $P\mapsto P|\dd$, where $P|\dd(A)=P(\dd A\dd)|_{-w}\dd$, and so it 
decomposes into eigenspaces $W_w^\pm(N)_{/R}$ for $\dd$ with eigenvalue~$\pm 1$. 
We call \emph{even} the polynomials in $W_w^+(N)$, and \emph{odd} those in $W_w^-(N)$. This is 
motivated by the fact that for $P\in W_w^+(N)$ the \emph{principal part} $P(I)$ is even, with~$I$ the 
coset of the identity, but not all components $P(A)$ are necessarily even.

\begin{remark} \label{r1}
The space $W_w(N)_{/R}$ is isomorphic to the space $\Symb_{\G_0(N)} V_w(R)$  
of modular symbols introduced by Ash and Stevens in~\cite{AS}, 
and by \cite[Prop.~4.2]{AS} we have a Hecke-equivariant isomorphism 
\be\label{3}
W_w(N)_{/R}\simeq H_c^1(\G_0(N),V_w(R)). 
\ee
The compactly supported cohomology group is that of the local system associated
to $V_w(R)$ on the modular surface $\G\bsh\H$, with $\H$ the upper half-plane. 

The module $V_w(N)_{/R}$ is simply the induced module $\ind_{\G}^{\G_1} V_w(R)$, 
so, via Shapiro's lemma, another way to interpret the isomorphism~\eqref{3} is:
 \[ W_w(N)_{/R}\simeq H_c^1(\G_1,V_w(N)_{/R}).\]  
Since $\G_1$ has only one cusp fixed by $T=US$ , the 
latter cohomology group can be identified with the set of $\G_1$-cocycles which 
are 0 on $T$ as in~\cite{H}, and the isomorphism takes a polynomial $P$ to the 
cocyle $\vp$ with $\vp(T)=0$, $\vp(T)=P$. See also~\cite[Sec. 2]{PP} and~\cite[Sec. 2.2]{P}. 

We will use the isomorphism~\eqref{3} in Section~\ref{S4}. We conclude that
the combinatorial description of $W_w(N)_{/R}$ that we use throughout 
Section~\ref{S2} gives us a way of studying ``Eisenstein classes'' 
in the compactly supported cohomology of the modular surface. 
\end{remark}
Throughout Section~\ref{S2} we are interested in the case $R=\C$, and we set 
$V_w(N)=V_w(N)_{/\C}$,  $W_w(N)=W_w(N)_{/\C}$. For a cuspform 
$f\in S_{k}(N)$, its period polynomial $\rho_f\in W_{k-2}(N)$ (which we sometimes 
denote by $\rho(f)$) is defined in~\cite{PP} by 
\be\label{2.2}
\rho_f(A)=\int_0^{i \infty} f|_k A(z) (X-z)^{k-2} dz, \
\quad \forall A\in \G\bsh \G_1,  \ee
where we set $f|_k \ss(z)=f(\ss z) (cz+d)^{-k}$ for $\ss=\sm abcd\in \GL_2^+(\R)$,
and $f|_k A$ is defined using any representative of the coset $A$.  This normalization 
of the stroke $|_k$ operator is chosen both to be compatible with the earlier operator 
$|_{-w}$ on period polynomials, and to avoid scaling factors in the action of Hecke 
operators--see the last equation in this subsection and~\eqref{10}.

The maps $\rho^\pm: S_{w+2}(N)\rar W_w(N)$, $f\mapsto \rho^\pm_f$ are injective,
and the Eichler-Shimura isomorphism can be restated as the following direct sum decomposition
\be\label{2.3}
W_w(N)\simeq \rho^+(S_{w+2}(N) ) \oplus \rho^-(S_{w+2}(N) ) \oplus C_w(N),
\ee
where $C_w(N)=\{P|1-S: P\in V_w(N), P|1-T=0\}$ is the \emph{coboundary subspace} \cite[Thm. 2.1]{PP}.
The dimension of $C_w(N)$ equals the dimension of the Eisenstein subspace of 
$M_{w+2}(N)$~\cite[Lemma 4.2]{PP}, and in Proposition~\ref{C2} we give
an explicit basis coming from Eisenstein series when the level~$N$ is square-free and $w>0$. 

To define the action of Hecke operators on period polynomials, let
$\M_n$ be the set of $2 \times 2$ integral
matrices of determinant $n$, and set $\ov{\M}_n:=\M_n/\{\pm 1\}$,
$\RR_n:=\Z[\ov{\M}_n]$. Let $\SS\subset \M_n$ be a double coset of $\G$, namely
$\SS=\G\SS\G$ and the number of right cosets $|\G\bsh \SS|$
is finite. The double coset $\SS$ acts on 
$f\in M_k(N)$ by 
\be \label{1} f|[\Sigma]=n^{k-1}\sum_{\sigma\in \G
\bsh\Sigma}f|_k\sigma.
\ee 
To define the corresponding action on period polynomials, we make
the following assumption on the double coset $\SS$:
 \be \label{eq_star} 
 \text{The map $\G\bsh\SS\longrightarrow \G_1\bsh\G_1\SS , \ \ \ $  $\G \sigma \mapsto
 \G_1\sigma  $  is bijective}\;, \ee
or equivalently $|\G\bsh\SS|= |\G_1\bsh\G_1\SS|$. For
$P\in V_w(N)$ and $M\in \M_n$, we define
\begin{equation}\label{eq_act1}
P|_{\SS}M(A)=\begin{cases} P(A_M)|_{-w}M  & \text{ if } M A^{-1} =A_M^{-1} M_A \text{ with } A_M\in\G_1, M_A\in\SS \\
          0 & \text{ otherwise.}
         \end{cases}
\end{equation}
Since both $M$ and $-M$ act in the same way, the
action of elements in $\ov{M}_n$ is also well defined,
and by linearity it extends to an action of elements in $\RR_n$. It
is not a proper action, but it is compatible with the action of
$\G_1$: for $g\in \G_1$, $M\in\M_n$, we have $P|_\SS gM=(P|g)|_\SS
M$, $P|_\SS Mg=(P|_\SS M)| g $.

Let $M_n^\infty$ be a system of representatives which fix
$\infty$ for the cosets $\G_1\bsh \M_n$, and let $T_n^\infty
=\sum_{M\in M_n^\infty} M\in \RR_n$. Let $T=US=\sm
1101$ be a generator of the stabilizer of~$\infty$. It was shown in~\cite{CZ} that
there exists $\wT_n\in \RR_n$ such that:
 \be \label{hecke} T^{\infty}_n(1-S)-(1-S)\wT_n \in (1-T)\RR_n, \ee
and in \cite{PP} we show that for $f\in
S_k(N)$ we have  $ \rho_{f|[\SS]}=\rho_f|_\SS \wT_n.$

\subsection{The Atkin-Lehner operator}\label{s2.2}
Let $\G=\G_0(N)$ and denote $w_N=\sm 0{-1}N0$, $\ss_N=\sm N001$.  The action of the 
Atkin-Lehner involution $W_N$ on modular forms 
$f\in M_k(N)$ is given by $f|W_N = N^{k/2} f|_k w_N$. It is related to
the action of the double coset
\[\Theta_N=\G w_N \G=\G w_N= \big\{ \sm abcd \in \M_N\;:\;  N|a, N|d, N|c \big\} \]
by $f|W_N=\dfrac{1}{N^{k/2-1}} f|[\Theta_N]$, with the latter action defined in~\eqref{1}.
We write 
$P|_{\Theta}\wT_N$ instead of $P|_{\Theta_N}\wT_N$ for the corresponding 
action on $P\in V_w(N)$ given by~\eqref{eq_act1}. 

\begin{lemma}\label{L1} \emph{(i)} We have a bijection
$\G \bsh \G_1\rightarrow \G_1\bsh \G_1 \ss_N\G_1$,
given for $A\in \G\bsh\G_1$ by
\[ A\mapsto K_A := \G_1 \ss_N A . \]
\emph{(ii)} We have $K_A=\G_1 \Theta_N A$. 
\end{lemma}

\begin{proof} (i)
Since $\G_0(N)=\G_1\cap \ss_N^{-1} \G_1 \ss_N$, the bijection follows 
from~\cite[Prop. 3.1]{Sh}.

(ii) This is immediate from
$\ \G_1 \Theta_N=\G_1\ss_N=\{\sm abcd\in \M_N : N|a, N|c \}$.
\end{proof}

For  $\widetilde{T}_n=\sum_{M\in
\M_n}c_M M\in \RR_n$  and a coset $K\in \G_1\bsh \M_n$, we let $\wT_n^{(K)}=\sum_{M\in K} c_M M$
be the part of $\wT_n$ supported on matrices in $K$. 
\begin{lemma}\label{L2}
Let $P\in V_w(N)$ and $\wT_N\in \RR_N$. Then 
 \[P|_{\Theta}\wT_N(A) = P|_{\Theta}\wT_N^{(K_A)} (A),\]
for $A\in \G\bsh \G_1$, where $K_A\in \G_1\bsh \M_N$ is
the coset defined in Lemma~\ref{L1}.
\end{lemma}
\begin{proof}By the definition~\eqref{eq_act1}, for 
$\wT_N=\sum c_M M\in \RR_N$ we have
\be\label{5} 
 P|_{\Theta}\wT_N(A)=\sum_{M\in \G_1\Theta_N A} c_M\cdot P(A_M)|_{-w}M
 =P|_{\Theta} \wT_n^{(K_A)}(A),
\ee
where $A_M\in \G \bsh\G_1$ is the unique coset such that 
$A_M MA^{-1}\subset \Theta_N$. The last equality follows from Lemma~\ref{L1} (ii).  
\end{proof}
\begin{example} \label{ex0}
For the identity coset $I$, we have $K_I=\G_1 \ss_N$ and we can take 
\[\wT_N^{(K_I)}=\sm N001.
\]
\comment{
(ii) For $A=\G U^2$ with $U^2=\sm 0{-1}1{-1}$, we have  $K_A=\G_1 \sm 110N$ 
and we can take
\[\wT_N^{(K_A)}=\sm 110N+\sm N011 .
\]   }
\end{example}

The space $W_w(1)$ contains the polynomial $1|_{-w}1-S=1-X^w$, which belongs 
to the coboundary subspace $C_w(1)$ defined in~\eqref{2.3} and corresponds 
to the Eisenstein series $E_k$, as we will see in~\S\ref{s2.4}. Therefore 
$W_w(N)$ contains the polynomial $P_0={\bf 1}|(1-S)$, with ${\bf 1}$ 
the constant polynomial 1 in each coset. We next determine its image under the Atkin-Lehner
operator, which will be used to determine a basis of
the ``Eisenstein part'' of $W_w(N)$ when $N$ is square-free. Denote by $(x,y)$ the greatest common 
divisor of $x,y\in\Z$.
\begin{proposition} \label{PAL} Let $w\ge 2$ be even, and  let $P_0={\bf 1}|(1-S)\in W_w(N)$. For every
$\wT_N$ satisfying \eqref{eq_star} we have
\[P_0|_{\Theta}\wT_N(A)= N_z^w-N_t^w X^w \]
where $A=\G \sm **zt$ and for $a\in \Z$ we let $N_a=N/(N,a)$.
 \end{proposition}
\begin{proof}
For a coset $K\in \G_1\bsh \M_n $, let $M_K\in K\cap M_n^\infty$ be a 
fixed representative fixing $\infty$, and let 
$T_n^\infty=\sum_{K\in \G_1\bsh \M_n} M_K$. Taking the part
of relation \eqref{eq_star} supported on matrices in $K$ we have
 $ (1-S) \wT_n^{(K)}-( M_{K}-M_{KS}S ) \in (1-T) \Q[K].$
Using Lemma~\ref{L2} we obtain
  \[P_0|_{\Theta}\wT_N(A)=1|_{-w}(1-S)\wT_N^{(K_A)}=1|_{-w}(M_{K_A}-M_{K_A S}S)
  =d^w-(d'X)^w   , \]
where we write $M_{K_A}=\sm ab0d$, $M_{K_A S}=\sm {a'}{b'}0{d'}$. One
checks that $d'=N/(b,d)$, and since 
$K_A=\G_1\sm ab0d=\G_1 \sm {Nx}{Ny}zt$ for $A=\G \sm xyzt$, we obtain 
$a=(N,z)$, $(b,d)=(N,t)$. 
\end{proof}

\subsection{Primitive spaces} \label{s2.3}
In this section we define subspaces $W_w (N)^\new\subset W_w(N)$ which 
contain the period polynomials of newforms in $M_{w+2}(N)$. 

To fix definitions, we first review some newform theory from~\cite{AL}. A modular form 
in $M_k(N)$ is called a Hecke eigenform if it is an eigenform
of all Hecke operators $T_n$ (including for primes $p|N$, which are 
called $U_p$ in~\cite{AL}), normalized to have the coefficient 
of~$q$ equal to~1. For a prime $p|N$, a cuspform $f$ is called $p$-new
if it is orthogonal with respect to the Petersson inner product to the space 
spanned by the image of the two embeddings 
$S_k(N/p)\hookrightarrow S_k(N)$ given by the identity and $f(z)\mapsto f(pz)$. 
The Hecke eigenforms which are $p$-new for all $p|N$  are called newforms. If $h$ is a
newform of level $M|N$, $M\ne N$, the oldspace associated to~$h$ is the span of $h(dz)$
for $d | (N/M)$, and the oldspaces together with the one dimensional spaces spanned by newforms 
give a decomposition of $S_k(N)$ into mutually orthogonal subspaces. 
There is also a notion of newforms for Eisenstein series~\cite{W},  but we only
need here the obvious fact that for $k\ge 4$ and square-free~$N>1$ all Eisenstein series in $M_k(N)$ are old. 

In this paper we use an algebraic characterization
of newspaces originally due to Serre. For 
$M | N$, let $\tr_{M}^N:M_k(N)\rightarrow M_k(M)$ be the trace map 
$\tr_{M}^N (f)=\sum_{\ss} f|_k \ss$, with the sum over a system of representatives
for the cosets $\G_0(N)\bsh\G_0(M)$.
For a prime $p|N$, the space of forms which are new at $p$ can be characterized as 
\[
M_k(N)^{p-\new}=\{f\in M_k(N):\; \tr^N_{N/p}(f)=\tr^N_{N/p}(f|W_N)=0 \},
\]
and we define the space of newforms $M_k(N)^{\new}=\cap_{p|N}M_k(N)^{p-\new}$, where $p$ runs
through the prime divisors of $N$. 
This agrees with the usual definition given above: see~\cite[Ch. VIII, Thm. 2.2]{L} for cusp forms
and~\cite[Prop. 19]{W} for Eisenstein series. 

Similarly, for $M|N$ we let 
$\tr_{M}^N:W_w(N)\rightarrow W_w(M)$ be the trace map:
$$\tr^N_M(P)(C)=\sum_{B\in \G_0(N)\bsh \G_0(M)} P(BC),$$ 
for all $C\in \G_0(M)\bsh\G_1$. This is compatible with the trace defined 
on the cuspidal space: for $f\in S_{w+2}(N)$ we easily see that 
$\rho (\tr^N_M f)=\tr^N_M \rho(f)$. We therefore define the new subspace of 
$W_w(N)$ by
$W_w(N)^\new=\cap_{p|N} W_w(N)^{p-\new}$, where for prime $p|N$ we define
\be \label{new}
W_w(N)^{p-\new}:=\{P\in W_w(N) : \tr^N_{N/p}(P)=\tr^N_{N/p}(P|_\Theta \wT_N)=0  \}.
\ee
Since the action of $\dd$ commutes with the trace map, we may define 
subspaces $W_w^{\pm}(N)^{p-\new}$ of $W_w^{\pm}(N)$. All these spaces can be
defined in the same way over an arbitrary ring~$R$, and we will need them in Section~\ref{s4.3}. 

\subsection{Extended period polynomials}\label{s2.40} We also need the space of extended 
period polynomials~$\wW_w(N)$, 
which contains the period polynomials $\wr(f)=\wr_f$ 
of arbitrary modular forms $f\in M_{k}(N)$  (we set $k=w+2$ throughout). We refer to~\cite[Sec. 8]{PP}
for the definition, and we only recall that for $f\in M_{k}(N)$, its extended 
period polynomial $\wr_f$ is given as in~\eqref{2.2}, with the integral 
regularized by replacing $f|A$ with $f|A-a_0(f|A)$, where $a_0(f|A)$
is the constant term in the Fourier expansion of~$f|A$. By~\cite[eq. (8.2)]{PP} 
we have\footnote{Here $k$ is even, but the formula is valid for all finite index subgroups of~$\SL_2(\Z)$,
when $k$ may also be odd.}
\be\label{2.8}
\wr_f(A)=(-1)^k a_0(f|A)\frac{X^{k-1}}{k-1}+a_0(f|AS)\frac{X^{-1}}{k-1}+
\sum_{n=0}^w (-1)^{w-n}\binom{w}{n}r_{n}(f|A) X^{w-n}
\ee
for $A\in \G_0(N)\bsh \G_1$, where $r_{m-1}(g)=(-1)^{m} \frac{\Gamma(m)}{(2\pi i)^{m}} L(g,m)$
is given in terms of the critical values at  $0<m<k$ of the $L$-function $L(g,s)$, extended by meromorphic
continuation. 

\begin{example}\label{ex4}
Since $L(E_k,s)=\zeta(s)\zeta(s-k+1)$, we obtain for $k\ge 4$ even:
\[ \wr^-\big(E_k\big)=-\frac{B_k}{2k}\cdot \frac{X^{k-1}+X^{-1}}{k-1}-\frac 12
 \sum_{0<n<k-2} \binom{k-2}{n-1} \frac{B_{n}}{n}\frac{B_{k-n}}{k-n} X^{n-1} 
 \in \wW_{k-2}^-(1)\,,  \]
 and $\wr^+\big(E_k\big)=\a_k (1-X^{k-2})\in W_{k-2}^+(1)$, 
 for $\a_k=\frac{(k-2)!}{2 (2\pi i)^{k-1}} \zeta(k-1)$~\cite[p. 240]{KZ}. 
\end{example}

The Hecke operators $\wT_n$ preserve $\wW_w(N)$,  acting as in Section~\ref{s2.1}, and we have 
\be \label{10}\wr_{f|[\SS]} = \wr_f|_\SS \wT_n \;, \ee
where $\SS$ is a double coset contained in $\M_n$ satisfying~\eqref{eq_star}. 
The space~$\wW_w(N)$ 
is preserved by the involution $\dd=\sm {-1}001$, and we denote its 
$\pm 1$ eigenspaces by $\wW_w^\pm(N)$. We define its new 
subspaces as in~\eqref{new}. The following proposition is a generalization of the Eichler-Shimura 
isomorphism to the space of extended period polynomials. 
\begin{proposition} \label{P8}
Let $w\ge 2$ be even. 
The two maps $$\wr^\pm: M_{w+2}(N)\rightarrow \wW_w^\pm(N), \quad f\mapsto \wr^\pm_f$$ 
are Hecke equivariant isomorphisms, and  
they map $M_{w+2}(N)^{p-\new}$ isomorphically onto $\wW_w^\pm(N)^{p-\new}$, for 
primes~$p|N$.
\end{proposition}
\begin{proof}That the two maps are isomorphisms is proved in \cite[Prop. 8.4]{PP}.
The second statement follows by using the characterisation of newforms above,
together with the compatibility of the two isomorphisms with the trace map 
$\tr=\tr^N_M$ and with the Atkin-Lehner involution: 
$
\wr_{\tr (f)}^\pm=\tr (\wr_f^\pm), \  \wr_{f|W_N}^\pm=N^{-w/2} \wr_f^\pm|_\Theta \wT_N.
$
\end{proof}
\begin{remark}
For $w=0$ and $N$ square-free the map $\wr^-$ is still an isomorphism, but $\wr^+$ is not 
unless $N$ is prime~\cite[Prop. 8.4]{PP}.
This is one of the reasons the weight 2 case is more delicate, and we avoid it in this paper. 
\end{remark}

\subsection{Period polynomials of Eisenstein series}\label{s2.4}
We now specialize $N$ to be square-free, and apply the results of 
the previous sections to determine the period polynomials of a basis of
Eisenstein series in~$M_k(N)$ for $k\ge 4$.
Let $D(N)$ denote the
divisors of $N$ and let $\e:D(N)\rar\{\pm 1\}$
be a system of Atkin-Lehner eigenvalues, namely  
$\e(a)\e(b)=\e(ab)$ if $(a,b)=1$. Since $E_k|W_d (z)=d^{k/2} E_k(dz)$, 
the linear combinations
$$E_{k,N}^{(\e)}(z):=\sum_{d|N} \e(d) d^{k/2} E_k (dz)\in M_k(N)$$ 
are eigenforms of $W_d$ with eigenvalue $\e(d)$ for all $d|N$, and they provide
a basis of the Eisenstein subspace (of dimension $2^{\omega(N)}$, 
with $\omega(N)$ the number of prime factors of~$N$). When $N=p$
is prime, we identify $\e$ with its value $\e(p)\in\{\pm 1\}$, and we recover
the Eisenstein series $E_{k,p}^{(\e)}$ from the introduction. 

When $N$ is square-free, we show next that  the extended polynomial $\wr^+_f$ is actually a 
period polynomial in $W_w^+(N)$ for all $f\in M_{w+2}(N)$, just like in the case $N=1$ of 
Example~\ref{ex4}. We also make more explicit the Eichler-Shimura isomorphism 
in Proposition~\ref{P8}, by determining an explicit basis of the coboundary subspace of $W_w(N)$. 
\begin{proposition}\label{C2}Let $N$ be square-free and let $k=w+2\ge 4$ be even.

\emph{(i)}  We have isomorphisms 
\[\rho^-:S_{k}(N)\overset{\sim}{\lrar} W_w^-(N), \quad
\wr^+: M_{k}(N)\overset{\sim}{\lrar} W_w^+(N),
\]
and, if $N>1$,  $\rho^+:S_k(N)^\new \overset{\sim}{\lrar} W_w^+(N)^\new$.   

\emph{(ii)} We have the following explicit version of the Eichler-Shimura isomorphism 
\[
W_{w}(N)=\rho^-(S_k(N))\oplus \rho^+(S_k(N))
\oplus_{\e} \C \wr^+ (E_{k,N}^{(\e)}) ,\]
where the period polynomials $\wr^+ (E_{k,N}^{(\e)})$ span the coboundary subspace $C_w(N)$. 
\end{proposition}
\begin{proof} (i)
The set of $N$ for which the map $\rho^-$ is an isomorphism is characterized 
in~\cite[Prop. 4.4]{PP}, and it includes square-free $N$.   
From the Eichler-Shimura isomorphism~\eqref{2.3}, we obtain that 
$\dim W_w^+(N)=\dim M_{w+2}(N)$. The latter is also equal to 
$\dim \wW_w^+(N)$,  so $\wW_w^+(N)=W_w^+(N)$, and Proposition~\ref{P8} 
implies that $\wr^+$ and $\rho^+$ in (i) are isomorphisms as well (the latter when $N>1$ since
$M_{w+2}(N)^\new=S_{w+2}(N)^\new$ in this case).   

(ii) The period polynomials  $\wr^+ (E_{k,N}^{(\e)})$ are Atkin-Lehner eigenforms with different eigenvalues,
so they are linearly independent. They belong to $C_w(N)$ since they are in the span of images of 
$\wr^+(E_k)\in C_w(N)$ under Atkin-Lehner involutions, which preserve $C_w(N)$.
\end{proof}
 
In the rest of this subsection we determine  $\wr^+ \big(E_{k,N}^{(\e)}\big)$ and 
the principal part of~$\wr^- \big(E_{k,N}^{(\e)}\big)$. Note that~\eqref{10} implies that 
$\wr^{\pm} \big(E_{k,N}^{(\e)}\big)$ is an eigenvector for the Hecke operators
$\wT_n$ with eigenvalue $\sigma_{k-1}(n)$ for $(n,N)=1$, as well as an 
eigenvector for all Atkin-Lehner operators. 

For $d|N$, the inclusion $M_k(d)\hookrightarrow M_k(N)$ corresponds to an
inclusion 
  \[ i_d^N:W_w(d)\hookrightarrow W_w(N) \] 
described as follows. For $A\in \G_0(N)\bsh \G_1$, write $A=BC$ with 
$B\in  \G_0(N)\bsh \G_0(d)$, 
$C\in \G_0(d)\bsh\G_1$. Then $(i_d^N P) (A)=P(C)$, and if $A=\G_0(N) \sm **zt$ 
then $C=\G_0(d)\sm **{(z,d)}{(t,d)}$.

For the Eisenstein series $E_k$ we have 
$ \wr^+(E_k)=\a (1-X^{k-2})\in W_{k-2}^+(1) $, with $\a=\a_k$ given explicitly in
Example~\ref{ex4}. 
\begin{proposition}\label{C1}
Let $N$ be square-free and let $k=w+2\ge 4$ be even.
For $\e:D(N)\rar\{\pm 1\}$ a system of Atkin-Lehner eigenvalues, 
we have 
\[
\wr^+ \big(E_{k,N}^{(\e)}\big)=\a \prod_{p|N}(1+\e(p)p^{-w/2})\cdot 
P^+\big(E_{k,N}^{(\e)}\big) 
\]
with $P^+\big(E_{k,N}^{(\e)}\big)\in W_{w}^+(N)$ given by
\[
P^+\big(E_{k,N}^{(\e)}\big)(A)=\e(N_z)N_z^{w/2}-\e(N_t)N_t^{w/2} X^w\in \Z[X]
\]
for $A=\G_0(N)\sm **zt$, where we recall that $N_a=N/(N,a)$.  
\end{proposition}
One can check directly that $P^+\big(E_{k,N}^{(\e)}\big)\in C_w(N)$, by writing it as 
$P_{w,N}^{(\e)}|1-S$ where   
$P_{w,N}^{(\e)}(A)=\e(N_z)N_z^{w/2}$ for a coset $A$ as above. One easily checks
$P_{w,N}^{(\e)}|1-T=0$.
\begin{proof} By~\eqref{10}, for each divisor $d|N$ we have
\[ \wr^+ (E_k|W_d)=d^{-w/2}\cdot [i_1^d\wr^+ (E_k) ]|_\Theta \wT_d \in W_w(d),\]
with the Atkin-Lehner involution~$W_d$ acting as in Section~\ref{s2.2}, yielding
\be\label{2.9}
\wr^+ \big(E_{k,N}^{(\e)}\big)= \sum_{d|N}\e(d)d^{-w/2} \cdot
i_d^N [i_1^d\wr^+ (E_k) |_\Theta \wT_d].
\ee
Note that $i_1^d \wr^+(E_k)=\a P_0   \in W_{k-2}(d)$, where 
$P_0$ is defined in Proposition~\ref{PAL}, and applying that proposition
we obtain: 
\[i_d^N [i_1^d\wr^+ (E_k) |_\Theta \wT_d](A)=\a\cdot (d_z^w-d_t^w X^w), 
\quad \text{ for } A=\G_0(N)\sm **zt .
\]
We now use the identity 
$ \sum_{d|N}\e(d)d^{-w/2} d_z^w = 
\e(N_z) N_z^{w/2}\cdot\prod_{p|N}(1+\e(p)p^{-w/2}) $.
\end{proof}

The computation of the principal part of $\wr^- \big(E_{k,N}^{(\e)}\big)$ is similar, 
using the formula for $\wr^-(E_k)\in \wW_{k-2}^-(1)$ in Example~\ref{ex4}.
\begin{proposition}\label{p2.10}
Let $N$ be square-free, let $k=w+2\ge 4$ be even, and 
let $\e:D(N)\rar\{\pm 1\}$ be a system of Atkin-Lehner eigenvalues. For the 
identity coset $I$ we have 
\[ 
\wr^- \big(E_{k,N}^{(\e)}\big)(I)=\sum_{d|N} \e(d) d^{-w/2}\cdot  
\wr^-(E_k)|_{-w} \sm d001 .
\]
\end{proposition}
\begin{proof} 
Apply~\eqref{2.9} written for the odd part in terms of $\wr^-(E_k)$,  
and use Lemma~\ref{L2} (which is easily seen to hold for extended polynomials) 
together with Example~\ref{ex0}.     
\end{proof}

\subsection{Trace maps.} \label{s2.5}
We end this section by determining the behavior of $E_{k,N}^{(\e)}$ under trace maps. 
Let $N=Mp$ be square-free with $p$ prime.  We can restrict a system of 
Atkin-Lehner eigenvalues~$\e:D(N)\rar\{\pm 1\}$ to~$D(M)$, and  
apply~\eqref{2.12} to $E_{k,N}^{(\e)}= E_{k,M}^{(\e)}|(1+\e(p)W_p)$.
Using $\wr^\pm (\tr^N_M f)=\tr^N_M  \wr^\pm(f)$ for $f\in M_k(N)$, 
we obtain
\be\label{2.10}
\begin{aligned}
\tr^N_M P^+(E_{k,N}^{(\e)})= \big(1+\e(p)p^{k/2}\big)\cdot P^+(E_{k,M}^{(\e)}) \\
\tr^N_M \wr^- (E_{k,N}^{(\e)})= \big(p^{k/2}+\e(p)\big)\big(p^{1-k/2}+\e(p)\big)\cdot
\wr^-(E_{k,M}^{(\e)})
\end{aligned}
\ee
where in the first equation we used Proposition~\ref{C1}.

\section{Proof of Theorems~\ref{T1} and \ref{T3}}
\label{s3}
Let $N=Mp$ with $p\nmid M$ as in the introduction, 
let $g\in M_k(M)$ be a newform of level $M$ with Fourier coefficients $\lambda_n(g)=\lambda_n$, 
and let $g_p^{(\e)}$ be defined as in~\eqref{1.2} for $\e\in \{\pm 1\}$. This includes 
the case $M=1$ and $g=E_k$ in Theorem~\ref{T1}. Let $\ell$ be a prime satisfying 
$\ell|\lambda_p+  \e p^{k/2-1}(p+1)$, which covers the ``new at $p$'' condition in both theorems
(see~\eqref{2.12}). 

We first observe that if $f\in S_k(N)^{p-\new}$ is a Hecke eigenform with 
 eigenvalue $\e$ under the Atkin-Lehner involution $W_p$, then $\lambda_{p}(f)=-\e p^{k/2-1}$ by~\cite[Thm. 3]{AL}, 
so\footnote{Both  $f$ and  $g_p^{(\e)}$  have Euler products,  so it is enough to check
the congruence for Fourier coefficients of prime index.}
 \[\lambda_{p}(f)\equiv \lambda_{p}(g_p^{(\e)})=\lambda_{p}+\e p^{k/2}
\pmod{ \lambda_p+  \e p^{k/2-1}(p+1)  }.
\]
That is, the congruences in both theorems hold at $p$ because of the above assumption on~$\ell$,
and therefore it is enough to check that there exists such an $f$ with Hecke eigenvalues 
 $\lambda_n(f)\equiv \lambda_n\pmod \I$ for $(n ,p)=1$.

Let~$R$ be a finite extension of~$\Z_\ell$ containing the coefficients of 
all Hecke eigenforms in $S_k^{(\e)}(N)$.  Let $\pi$ be a uniformizer in $R$
and $F=R/\pi R$ the residue field, and set $w=k-2$. Since $\ell>k+1$, $\ell\nmid N$, 
by Theorem~\ref{T2} the reduction map 
$ W_w(N)_{/R}^{p-\new}\rar W_w(N)_{/ F}^{p-\new} $
is surjective, and both theorems follow from the Deligne-Serre lifting lemma, once
we produce an element in $W_w(N)_{/ F}^{p-\new}$ which is an eigenvector for the Hecke 
operators of index coprime to $p$ and for the Atkin-Lehner involution $W_p$, 
with eigenvalues congruent to those of~$g_p^{(\e)}$.  
The Deligne-Serre lemma would then provide a system of  Hecke and Atkin-Lehner 
eigenvalues on $W_w(N)_{/R}^{p-\new}$ congruent to those of~$g_p^{(\e)}$. 
By Proposition~\ref{P8} we then conclude the existence of a Hecke eigenform 
$f\in M_k^{(\e)}(N)$ which is $p$-new, satisfying the desired congruence. 
But there are no $p$-new Eisenstein series in $M_k(N)$, as $p\nmid M$, 
so the form $f$ must be a cusp form.   Note that since $f$ is $p$-new, 
it is automatically an eigenform of $T_p$, being in the oldspace of a newform  
of level $pM'$ with $M'|M$.

To construct the desired finite period polynomial in both theorems we proceed as 
follows. 

$\bullet$ Theorem~\ref{T1}, $\ell|p^{k/2}+\e$. From~\eqref{2.10} and the definition of 
the newspace in~\eqref{new} we obtain that $P^+(E_{k,p}^{(\e)}) \pmod \ell$ belongs  to
$W_{w}(p)^\new_{/ \FF_\ell}$ (note that it is already an eigenform for the 
Atkin-Lehner operator~$|_\Theta \wT_p$ acting on period polynomials). It is nonzero 
modulo~$\ell$ since
$P^+\big(E_{k,p}^{(\e)}\big)(I)=1-\e p^{w/2}X^w$, by Proposition~\ref{C1}. 

$\bullet$ Theorem~\ref{T1}, $\ell\nmid p^{k/2}+\e$. 
Since $\ell$ divides the numerator of $B_k/k$, formula~\eqref{2.8} 
shows that
$\wr^-\big( E_{k,p}^{(\e)}\big)\pmod \ell$ 
belongs to $W_{w}^-(p)_{/ \FF_\ell}$. Since $\ell|p^{w/2}+\e$, it follows 
from~\eqref{2.10} that $\wr^-\big( E_{k,p}^{(\e)}\big)\!\! \pmod\ell$ actually belongs 
to $W_{w}(p)^\new_{/ \FF_\ell}$. From Proposition~\ref{p2.10}  we have 
\[
\wr^-\big(E_{k,p}^{(\e)}\big)(I) \equiv  \frac 12
 \sum_{0<n<w} \binom{k-2}{n-1} \frac{B_{n}}{n}\frac{B_{k-n}}{k-n}(1-p^{n-1}) X^{n-1}
\pmod \ell .
\]
The denominators in this formula have prime factors which are smaller than 
$\ell$, since $\ell>k-2$, and the extra assumption ensures that the 
previous element is nonzero. 

$\bullet$ Theorem~\ref{T3}. Let $P^\pm(g)\in W_w(M)$ 
be the multiples of $\rho^\pm(g)$ which are normalized as in the paragraph
following Theorem~\ref{T3}, so their coefficients belong to~$K_g$ by  
a well-known rationality result, e.g. ~\cite[Prop. 5.11]{PP}.
The condition $k\ge 6$ is required to guarantee that 
the coefficient of $X$ in $P^-(g)(I)$ is nonzero, being proportional to the 
value at $s=k-2$ of the $L$-function $L(s,g)$. 
From~\eqref{2.12} we obtain as before
\[\tr^N_M \big( P^\pm(g_p^{(\e)}) \big)= (1+p+\e p^{-k/2+1}\lambda_p)\cdot P^\pm(g),\] 
where $P^\pm\big(g_p^{(\e)}\big):= i_M^N P^\pm(g)+\e p^{-w/2}i_M^NP^\pm (g)|_\Theta \wT_p$  
is a multiple of $\rho^{\pm}\big(g_p^{(\e)}\big)$. By assumption, either 
$P^+(g)$ or $P^-(g)$ has denominators coprime to $\I$, so the same is 
true about at least one of $P^\pm\big(g_p^{(\e)}\big)$. It follows that one of
$P^\pm\big(g_p^{(\e)}\big)\pmod \I$ belongs to  $W_w^\pm(N)^{p-\new}_{/F}$, and we fix
this choice of sign.

To see that it is nonzero, we evaluate it on the identity coset $I$. 
Setting $P=P^\pm(g)(I)$ for the choice of sign above, we obtain by 
the definition~\eqref{2.2} and a change of variables: 
$$P^\pm\big(g_p^{(\e)}\big)(I)(X)=P(X)+\e p^{-w/2} P(pX). $$ 
Due to the normalization of $P$, the constant term is $1+\e p^{-w/2}$ for 
$P^+$ and the coefficient of~$X$ is $1+\e p^{-w/2+1}$ for $P^-$, which are 
nonzero mod $\ell$ in either case by assumption~\eqref{1.5}. 

\begin{remark}\label{r2} We found numerically that the extra assumption 
$\ell\nmid B_nB_{k-n}(p^{n-1}-1)$ in Theorem~\ref{T1} is not needed for 
weights $k\le 6\cdot 10^4$, by considering $n=2$ and $n=4$. 

More precisely, assume $\ell\nmid p^{k/2}+\e$. It follows from~\eqref{1.6}
that $\ell|p^{k/2-1}+\e$, $\ell|B_k$, and $\ell\nmid p-1$, so  
the assumption is satisfied for $n=2$, unless $\ell|B_{k-2}$. For 
$k \le 6\cdot 10^4$ there are only two pairs $(k,\ell)$ with $k$ even 
and $\ell>k-2$ prime, such that $\ell$ divides the numerator of both~$B_k$ 
and~$B_{k-2}$, namely $(92, 587)$ and $(338,491)$.\footnote{Using PARI~\cite{Pari}, it took
about 50 minutes on a laptop to check the range $5\cdot 10^4\le k\le 6\cdot 10^4$.} 
Note that in both cases 
$3\nmid \ell-1$, and since $\ell|p^{\ell-1}-1$ it follows that $\ell\nmid p^3-1$ 
(otherwise we would have $\ell|p-1$, contradicting the assumption). It follows 
that $\ell\nmid B_4 B_{k-4}(p^3-1)$ for those two values of $k$, so the assumption 
is satisfied for~$n=4$.

Note that the same assumption, without the factor $p^{n-1}-1$, appears in 
Haberland's proof of the Ramanujan congruence~\eqref{e0}~\cite[Sec. 5.2]{H}. 
There the assumption guarantees the nonvanishing of the reduction mod $\ell$ of an 
Eisenstein cocycle in $H^1(\G_1, V_w(R))$ associated to
$\wr^-(E_{k})$. 
\end{remark}

\section{Surjectivity of reduction maps on spaces of period polynomials}\label{S4}

In this section we use the isomorphism $W_w(N)_{/R}\simeq H_c^1(\G_0(N),V_w(R))$ 
in Remark~\ref{r1} to prove Theorem~\ref{T2}. 

\subsection{Surjectivity of reduction on the whole space}
\label{s4.1}

We first need a lemma computing the dimension of the cohomology of $\G_1=\SL_2(\Z)$,
for which we start in greater generality. 
Let $V$ be a right $\G_1$-module, and assume, as it will always be the case, 
that $-1\in\G_1$ acts trivially on $V$. Therefore the cohomology groups we consider are the same when
replacing $\G_1$ by $\ov{\G}_1=\PSL_2(\Z)$.

Since  $\ov{\G}_1$ is a free product of its subgroups $G_2$ and $G_3$ generated by $S$ and
$U$, the Mayer-Vietoris exact sequence in group cohomology 
\cite[Sec. VII.9]{Br} gives 
\be\label{MV}\begin{aligned}
0&\rar H^0(\ov{\G}_1, V)\rar H^0(G_2, V)\oplus H^0(G_3, V)\rar H^0(G_2\cap G_3, V)\rar\\
&\rar H^1(\ov{\G}_1, V)\rar H^1(G_2, V)\oplus H^1(G_3, V)
\end{aligned}
\ee \vspace{-3mm}
\begin{lemma}\label{P1}
Assume that the $\G_1$-module $V$ is defined over a 
field~$F$ of characteristic $\ch(F)\ne 2,3$. We have 
\[
  \dim H^1(\ov{\G}_1, V) =\dim V-\dim H^0(G_2, V)-\dim H^0(G_3, V)+\dim H^0(\ov{\G}_1, V). 
\]
\end{lemma}
\begin{proof}The assumption $\ch (F)\ne 2,3$  
implies that $H^j(G_i, V)=0$ for $j\ge 1$, $i=2,3$,
so the conclusion immediately follows from the Mayer-Vietoris sequence.
\end{proof}

Let $R$ be a discrete valuation ring with residue field $F$. The surjectivity
of more general reduction maps on compactly supported cohomology was proved 
by Hida~\cite[Eq. (1.16)]{Hi} for congruence groups with no elliptic elements, 
by a geometric argument. We give an algebraic proof here, valid for groups with 
elliptic elements as well. 

\begin{proposition}\label{P2} Let $w\ge 0$ be even. If the residue field $F$ has characteristic
$\ell>w$, $\ell\ne 2,3$, then the reduction  map 
$\ W_w(N)_{/R} \rightarrow W_w(N)_{/F}  $
is surjective. 
\end{proposition}
\begin{proof} 
 The reduction map is a composition 
\[
W_w(N)_{/R} \twoheadrightarrow W_w(N)_{/R}\otimes F \hookrightarrow W_w(N)_{/F},
\]
with the first map surjective and the  second map injective. 
Therefore surjectivity reduces to the equality
of the dimensions of the last two spaces as vector spaces over $F$. 

Since $R$ is a DVR, $W_w(N)_{/R}$  is a free $R$-module so 
$$\dim W_w(N)_{/R}\otimes F =\rk W_w(N)_{/R}=\dim W_w(N)_{/\C}
=\dim H^1(\G_0(N),V_w(\C) ),$$
where the second equality follows from the fact that $W_w(N)_{/\Z}$ is a sublattice
of $W_w(N)_{/\C}$, and the third follows
from  $W_w(N)_{/\C}\simeq H_c^1(\G_0(N),V_w(\C))$ and
Poincar\'e duality over $\C$.

The hypothesis $\ell>w$ implies that the $\G_1$-invariant pairing on $V(F)$ 
induced by the natural $\G_1$-invariant pairing on $V_w$ is nondegenerate, so 
$V^*(F)\simeq V(F)$. By Poincar\'e duality~\cite[Lemma 1.4.3]{AS}, 
it follows that 
$$\dim W_w(N)_{/F}= \dim_{F} H^1_c(\G_0(N), V_w(F) )= 
\dim_{F} H^1(\G_0(N), V_w(F)).$$  
Lemma~\ref{P1} shows that $\dim_{F} H^1(\G_1, V_w(N)_{/F})$ is 
the same for all fields $F$ with $\ch(F)\ne 2,3$,
where $V_w(N)_{/F}$ is the induced module $\ind_{\G_0(N)}^{\G_1} V_w(F)$. 
Applying this to  the residue field $F$ and to $\C$, and using the 
Shapiro lemma, we conclude from the last two displayed equations that 
$ W_w(N)_{/R}\otimes F =  W_w(N)_{/F}$, as they have the same dimension.
\end{proof}

\subsection{Surjectivity of reduction on the $p$-new subspace}\label{s4.3}

For $N=Mp$ with $p$ prime, in~\eqref{new} we have defined 
$$ W_w(N)_{/R}^{p-\new}=\ker\big(\beta: W_w(N)_{/R}\rar W_w(M)^2_{/R} \big) $$
where $\beta(P)=(\tr^N_{M} P, \tr^{N}_{M} P|_\Theta \wT_N)$. We recall 
that the operator $|_\Theta \wT_N$ on $W_w(N)$ corresponds to the Atkin-Lehner involution 
$W_N$ on $M_{w+2}(N)$ as in Section~\ref{s2.2}. 

Let $\G_0(M)':=\sm p001^{-1}\G_0(M) \sm p001$ for $p\nmid M$. We have 
$\G_0(M)\cap\G_0(M)'=\G_0(Mp)$, and we consider the sum of restriction maps: 
\be \label{12}
\a: H^1(\G_0(M), V_w(F))\oplus H^1(\G_0(M)', V_w(F))\ 
\rar H^1(\G_0(Mp), V_w(F)). \ee
We prove below that the maps $\a$ and $\b$ are essentially Poincar\'e dual to each other. 

\begin{proposition}\label{P5} Let $R$ be a discrete valuation ring with residue field~$F$ of 
characteristic~$\ell$. Let $w\ge 0$ be even, let $N=Mp$ with $p$ prime, $p\nmid M$,
and assume that $\ell>w$, $\ell\nmid 6$. The following are equivalent:

\emph{\phantom{ii}(i)} The reduction  map 
$\, W_w(N)_{/R}^{p-\new}\rar W_w(N)_{/F}^{p-\new} \,$
is surjective. 

\emph{\phantom{i}(ii)} The map $\beta: W_w(N)_{/F}\rar W_w(M)^2_{/F} $ is surjective.

\emph{(iii)} The map $\a$ given by~\eqref{12} is injective.
\end{proposition}

\begin{proof} (i) $\Leftrightarrow$ (ii):
The space $W_w(N)_{/R}^{p-\new}$ is free over $R$, of rank equal to the 
dimension of $W_w(N)_{/\C}^{p-\new}$. Since $R$ is a DVR, this rank also 
equals the dimension of the reduction $W_w(N)_{/R}^{p-\new}\otimes F$, so 
(i) is equivalent to $\dim W_w(N)_{/F}^{p-\new}=\dim W_w(N)_{/\C}^{p-\new}$.
On the other hand (ii) is equivalent to 
\[\begin{aligned}
     \dim W_w(N)_{/F}^{p-\new}&=\dim W_w(N)_{/F}-2 \dim W_w(M)_{/F} \\
     &=\dim W_w(N)_{/\C}-2 \dim W_w(M)_{/\C}
  \end{aligned}
\]
where the second equality follows from Proposition~\ref{P2}. It remains to show that
the latter difference equals $\dim W_w(N)_{/\C}^{p-\new}$, that is that the map $\beta$ 
is surjective over $\C$. For this we follow the proof of surjectivity given below,
which works over $\C$ with little change. Indeed
 the same proof shows that (ii) and (iii) are equivalent over~$\C$ as well, and the proof
 of (iii) over $\C$ is the same as that of Proposition~\ref{P3} below, but without needing Lemma~\ref{L3} and
Proposition~\ref{P4}. Instead, the fact that $H^1(\DD_M, V_w(\C))$ vanishes, where $\DD_M$ is the principal 
congruence subgroup of level $M$ of $\PSL_2(\Z[1/p])$, is a consequence of Cor. 2 to Thm. 5 in~\cite{Se}.

(ii) $\Leftrightarrow$ (iii): Since the residue field $F$ and $w$ are fixed,
we write $V=V_w(F)$. We have $W_w(N)_{/F}\simeq H_c^1(\G_0(N), V)$, 
and $\tr_M^N: W_w(N)_{/F}\rar W_w(M)_{/F}$ corresponds to the corestriction
map on the compactly supported cohomology groups, while the map 
$P\mapsto P|_\Theta \wT_N$ on $W_w(N)_{/F}$ corresponds to the Atkin-Lehner 
operator $[\Theta_N]$ on $H_c^1(\G_0(N), V)$. Therefore the first part in the 
diagram below is commutative. 
\[\xymatrix@C=0.7em{
W_w(N)_{/F}\ar[rrr]^-{\simeq} \ar[d]_{\beta} &&& 
H_c^1(\G_0(N), V)\ar[d]_{(\cor, \cor\,\circ\,[\Theta_N])}  
  & \times & H^1(\G_0(N), V) \ar[rr] && F \\
W_w(M)_{/F}^2\ar[rrr]^-{\simeq} &&&  
H_c^1(\G_0(M), V)^2  & \times & H^1(\G_0(M), V)^2
\ar[u]_{\res +  [\Theta_N] \,\circ\,\res} \ar[rr] && F}
\]
The second part is given by Poincar\'e duality, taking into account 
that $V\simeq V^*$ since $\ell>w$. For 
$\vp \in  H^1_c(\G_0(N), V)$, $\psi \in H^1(\G_0(M), V)$ and
$\vp' \in  H^1(\G_0(N), V)$ we have~\cite[Sec. 6.3]{Hi1}
$$\la \cor \vp, \psi\ra_M=\la \vp, \res \psi \ra_N, \quad
\la \vp|[\Theta_N], \vp'\ra =\la \vp, \vp'|[\Theta_N] \ra . $$
Since Poincar\'e duality is a perfect pairing, it follows that~$\beta$ 
is surjective if and only if the rightmost map is injective. 

Let $c_p: H^1(\G_0(M), V)\rar H^1(\G_0(M)', V)$ be conjugation
by $\sm p001$. We easily check that the following diagram commutes
\[
  \xymatrix{
H^1(\G_0(N), V)\ar[rr]^{[\Theta_N]}_{\simeq} && H^1(\G_0(N), V) \\
 H^1(\G_0(M), V)\ar[u]^{\res}\ar[r]^{[\Theta_M]}_{\simeq} &
 H^1(\G_0(M), V)\ar[r]^{c_p}_{\simeq} &
 H^1(\G_0(M)', V)\ar[u]^{\res}
 }
\]
which shows that the rightmost map in the diagram differs from $\a$ only
by the isomorphism $c_p\;\circ\;[\Theta_M]$ in the second factor, so rightmost map is injective if 
and only if $\a$ is injective. 
\end{proof}
\comment{
In the second vertical map, $\Theta_p=\G_0(N)w_p \G_0(N)=\G_0(N)w_p$ is 
the double coset corresponding 
to the Atkin-Lehner involution $W_p$, where $w_p=\sm {pa}b{Nc}{pd}$ with 
$a,b,c,d\in \Z$, $pad-Mbc=1$. Its action 
on $H^1(\G_0(N), V)$ is given on cocycles by $\vp|[\Theta_p] (g)=
\vp(w_pg w_p^{-1})|_{-w} w_p$. Since 
$w_p=\g\sm p001$ with $\g\in\G_0(M)$, we have $\G_0(M)'=w_p^{-1} \G_0(M)w_p$,
so we can view the second map also as the composition 
$H^1(\G_0(M), V)\stackrel{\simeq}{\longrightarrow} 
H^1(\G_0(M)', V)\stackrel{\res}{\longrightarrow} H^1(\G_0(N), V)$,
with the isomorphism given by conjugation. Therefore the second vertical map is the
map $\a_p$. 
}

We are reduced to proving the injectivity of $\a$, for which we use the 
ingredients of Ihara's lemma~\cite[Lemma 3.2]{Ih}. Our proof is inspired by
the proof of similar statements for parabolic cohomology given 
in~\cite{Ri,Di}. First we need an easy lemma.
\begin{lemma}\label{L3}Let $w\ge 0$, and
consider the $\SL_2(\Z)$-module $V_w(F)$, with $F$ a field of characteristic 
$\ell>w+1$. Let $u=\sm 1a01\in \SL_2(\Z)$ with $\ell\nmid a$. We have:
 
\emph{\phantom{i}(i)} $\im(1-u)=V_{w-1}(F)$ (setting $V_{-1}(F)=\{0\}$), and $\ker(1-u)=V_0(F)$; 

\emph{(ii)} $\ker(N)=V_w(F)$, where $N:=1+u+\ldots+u^{\ell-1}\in \Z[\SL_2(\Z)]$ acts by linearity
on $V_w(F)$. 
\end{lemma}

\begin{proof}
(i) The matrix of $1-u$ in the basis $1,X, \ldots, X^w$ is upper triangular,
with 0's on the diagonal and elements $\binom{n}{i}a^i$ with 
$w\ge n \ge i$ above the diagonal, which are invertible in $F$ since $\ell>w$.
It follows that $\ker(1-u)=F$, and since $\im(1-u)$ 
is contained in $V_{w-1}(F)$ it must be the entire subspace. 

(ii) Since $u^{\ell}$ acts as identity on $V_w(F)$, we have 
$\im(1-u)\subset \ker(N)$. By (i) we only have to  
check that $X^w\in \ker(N)$, namely that the polynomial
$$Q_w(X)=X^w+(X+1)^w+\ldots+(X+\ell-1)^w$$   
is identically 0 in $\FF_\ell[X]$. We prove this by induction on $w$. For $w=0$ the statement
is clear, and assuming it true for $w-1\ge 0$ and taking derivatives we have 
$Q_w'=wQ_{w-1}=0$. Therefore $Q_w$ is constant and we only have to prove that its constant
term vanishes, which we leave as an exercise. Note that if $\ell=w+1$ we have 
$Q_w(0)=-1$. 
\end{proof}
\begin{remark}\label{r4.5}
For $\ell=w+1$, part (ii) in the lemma is no longer true (from the proof we see that  
$\ker(N)=V_{w-1}(F)$, $\im(N)=V_0(F)$ in this case).  For this reason, the case $\ell=w+1$ is 
not included in the next proposition and in Theorem~\ref{T2}. 
\end{remark}
\begin{proposition} \label{P3} Let $w\ge 0$ be even, let $p\nmid M$ be prime and let
$V=V_w(F)$ with $F$ a field of characteristic $\ell>w+3$. Assume $\ell\nmid pM$, and if  
$w=0$ assume also that $\ell\nmid \vp(M)$.  Then the restriction map 
$$\a:H^1(\G_0(M), V)\oplus H^1(\G_0(M)', V)\ \rightarrow H^1(\G_0(Mp), V) $$ 
is injective (if $\ell= w+3>3$, its kernel is one-dimensional). 
\end{proposition}
\begin{proof}
Let $\G(M)$ be the principal congruence subgroup of level $M$ of
$\PSL_2(\Z)$, and $\G(M)':=\sm p001 ^{-1}
\G(M) \sm p001$. Their intersection is $\G(M)\cap \G_0(p)\subset \G_0(Mp)$, and 
we have a commutative diagram of restriction maps\footnote{The cohomology groups 
do not change upon replacing $\G_0(M)$ by its projectivization 
$\G_0(M)/ \{\pm 1\}$, as $-1$ acts trivially on $V.$}
\[\xymatrix{
H^1(\G_0(M), V)\oplus H^1(\G_0(M)', V) \ar[d]^{\res} \ar[r]^-{\a} & H^1(\G_0(Mp), V)
\ar[d]^{\res}\\
H^1(\G(M), V)\oplus H^1(\G(M)', V)\ar[r]^-{\g} & H^1(\G(M)\cap \G(M)',V) 
}\] 
The first vertical restriction is injective: using the inflation-restriction exact sequence
it is enough to show that $H^1\big(\G_0(M)/\G(M), V^{\G(M)}\big)=0$.
When $w>0$, the space of invariants  $V^{\G(M)}$ is trivial, since the invariants 
under $\sm 1M01$ are the constant polynomials, by Lemma~\ref{L3}, while the
only constant invariant under $\sm 10M1$ is 0. If $w=0$, we use that the quotient  
$\G_0(M)/\G(M)\simeq(\Z/M\Z)^*\ltimes(\Z/M\Z)$ has order $M\vp(M)$, which 
is coprime to $\ell$ by assumption, so the cohomology group vanishes (it is here
that we use the extra assumption $\ell\nmid \vp(M)$ when $w=0$). 

Therefore to show that $\a$ is injective it is enough to show that $\g$ is injective.
For that we use the following two ingredients of Ihara's lemma~\cite[Lemma 3.2]{Ih}. Let $\DD_M$ be the principal 
congruence subgroup of level $M$ of $\PSL_2(\Z[1/p])$.

\begin{enumerate}
   \item [(I1)] The group $\DD_M$ is the free product of $\G(M)$ and $\G(M)'$ with
amalgamated subgroup $\G(M)\cap \G(M)'$;
   \item [(I2)] The group $\DD_M$ is the normal closure of $\sm 1M01$ in $\DD_1$.     
\end{enumerate}

By (I1), the Mayer-Vietoris exact sequence gives:
\[
\cdots \rar H^1(\DD_M, V)\rar H^1(\G(M), V)\oplus
H^1(\G(M)', V)\ \stackrel{\g}{\longrightarrow} H^1(\G(M)\cap \G(M)',V)
\rar \cdots\,,
\]
so $\ker \g=H^1(\DD_M, V)$.
The inflation-restriction exact sequence for the normal 
subgroup $\DD_{\ell M}\subset \DD_M$ gives:
\[
0\rar H^1(\PSL_2(\FF_\ell), V^{\DD_{\ell
M}})\stackrel{\mathrm{inf}}{\longrightarrow} H^1(\DD_M, V)
\stackrel{\mathrm{res}}{\longrightarrow} H^1(\DD_{\ell M},
V)^{\DD_M/\DD_{\ell M}}.
\]
We have $V^{\DD_{\ell M}}=V$,
and we will see in Proposition~\ref{P4} below that 
$H^1(\PSL_2(\FF_\ell), V)$ vanishes under the assumption $\ell\ne w+3$. Therefore 
to show $H^1(\DD_M, V)=0$ it suffices to prove that the restriction map is 
identically 0. 

For this, we use (I2),  that is the fact that $\DD_{\ell M}$ is generated by elements of the
form $g \sm 1{\ell M} 01 g^{-1}$, $g \in \DD_1$. Let
$g \sm 1{\ell M} 01 g^{-1}=v^{\ell}$, with $v=g u g^{-1}\in \DD_M$ for 
$u=\sm 1M01$. For any cocycle  $\vp\in Z^1(\DD_M,V)$,   
we have  
\[\vp(v^{\ell})=\vp(v)|(1+v+\ldots
+v^{\ell-1})=(\vp(v)|g)|(1+u+\ldots +u^{\ell-1})|g^{-1}=0  \] 
by Lemma~\ref{L3} (ii) (here we use $\ell\ne w+1$). We conclude $\res \vp=0$.
\end{proof}

It remains to prove the vanishing of a finite cohomology group,
which was essentially proved in~\cite[Theorem 1.5.3]{KPS}. Since the proof
there was only sketched, and since we need a slightly more general
ground field, we fill in the details below.  
The entire structure of the cohomology ring is also determined in~\cite{Te}.
\begin{proposition}\label{P4} Let $w\ge 0$ be even, and let~$F$ be a field 
of characteristic~$\ell>w$, $\ell>3$. Then the cohomology 
group $H^1(\SL_2(\FF_\ell), V_w(F))$ vanishes,
unless $\ell=w+3$ when it is one-dimensional.
\end{proposition}
\begin{proof} We set $\G=\SL_2(\FF_\ell)$, $V=V_w(F)$, and 
consider more generally $H^n(\G, V)$ for 
$n>0$. Let $B\subset \G$ be the Borel subgroup of upper triangular matrices,
so that $|B|=\ell(\ell-1)$. The composition $
\cor\circ \res: H^n(\G, V)\rar H^n(B, V) \rar H^n(\G, V)$
of the restriction and transfer maps equals multiplication by the index
$[\G:B]=\ell+1$~\cite[Ch. III, Prop. 9.5]{Br}, which is an isomorphism of  
vector spaces over $F$ as the index is coprime to $\ell$. The composition
$\res\circ \cor$ also equals multiplication by  $[\G:B]$, 
by the Cartan-Eilenberg stability criterion~\cite[Ch. XII Prop. 9.4]{CE}. To
apply the criterion, we have to check that 
$H^n(B, V)$ consists of stable cohomology classes, namely 
for all $x\in \G$ we have a commutative diagram
\[\xymatrix@C0.5em{
     H^n(B,V)\ar[rr]^-{\simeq} \ar[rd]_{\res} & & H^n(xBx^{-1},V)\ar[ld]^{\res} \\
     & H^n(B\cap xBx^{-1},V) &
     }
\]
with the horizontal isomorphism being conjugation by $x$. To prove the commutativity, note 
that $B\cap xBx^{-1}$ is either $B$ (when $x\in B$), or the diagonal subgroup $T$ (when $x\notin B$). 
In the first case the statement is trivial, while in the second it follows 
from the fact that $H^n(T, V)=0$, as $T$ is cyclic of order 
$\ell-1$ coprime to $\ell$. We conclude that $H^n(\G, V)\simeq H^n(B, V)$.

Let $U=\{\sm 1*01 \in \G \}$, which is normal in $B$ with $B/U\simeq T$, 
the diagonal subgroup. The  inflation-restriction exact sequence together with $H^n(T, V)=0$ for 
$n>0$ implies that 
\[ H^n(B, V)\simeq H^n(U, V)^{B/U}.\]

Since $U$ is cyclic generated by $u=\sm 1101$, its cohomology 
$H^n(U, V)$ equals $\ker N/\im(1-u)$ if $n$ is odd and $\ker(1-u)/\im N$
if $n>0$ is even \cite[p. 58]{Br}, where $N=1+u+\ldots+u^{\ell-1}: V\rightarrow V$ 
is the norm map. By Lemma~\ref{L3} we obtain 
$H^n(U, V)\simeq F$ for all $n\ge 0$ if $\ell\ne w+1$, and $H^n(U, V)=0$ if $\ell=w+1$ 
(see Remark~\ref{r4.5}). 

To compute the invariants under $B/U$ assume $n=1$, and let $\vp:U\rar V$ be 
the generator of $H^1(U, V)$ with $\vp(u)=X^w\notin \im(1-u)$.  
The group $B/U\simeq T$ acts on cocycles by 
$\vp|g (n)=\vp(gn g^{-1})|_{-w}g$ for $g\in B, n\in U$, so the class of~$\vp$ 
is invariant under $T$ if and only if 
$\vp(u)-\vp(t u t^{-1})|_{-w} t \in \im(1-u)=V_{w-1}$ for 
$t=\sm a00{a^{-1}}\in T$ (i.e. $\vp-\vp|t$ is a coboundary). Since $\vp(u)=X^w$,
this happens if and only if $a^{w+2}=1$ for all $a\in \FF_\ell^*$, 
i.e., if and only if $\ell-1|w+2$. Since $\ell>w$, $\ell\ge 5$, we conclude 
that $H^1(U, V)^{B/U}=0$ unless $\ell=w+3$, when it is one dimensional.
\end{proof}

\section{Numerical examples}\label{S5}

For the examples in this section, we computed spaces of period polynomials
and individual eigenforms using the software MAGMA~\cite{MGM}. More details 
are given in~\cite[Sec. 5.5]{PP}.

\begin{example} \label{ex1} We expect Theorem~\ref{T1} to hold for $\ell=k+1$ as well,
and we verified the congruence in many such cases. For example, let 
 $p=19$, $k=6$, $\ell=k+1=7$. Although $\ell=k+1$ divides once the denominator of 
$B_k$ (which always holds when $k+1$ is prime by the von Staudt-Clausen Theorem),
we have $\ell^3|p^{k/2}+1$, so $\ell$ divides the numerator of 
$\frac{B_k}{2k}(p^{k/2}+1)$. We find a congruence modulo $\ell$ between 
$E_{k,p}^{(+1)}$ and the newform $f\in S_k^{(+1)}(p)$ with $q$-expansion:
\[f=q-2q^2-q^3-28 q^4-24q^5+2q^6-167q^7+\ldots-361 q^{19}+\ldots\,.
\]
Note that in order to prove this congruence, and the ones below, it is enough to check that it holds 
for coefficients of prime index up to the Sturm bound $k(p+1)/12$.
\end{example}
\comment{
\begin{example} Let $p=47$, $k=24$, $\ell=103$. We have 
$\ell|p^{k/2}-1$, but also $\ell| \frac{B_k}{2k}$, so this case is not covered 
by~\cite{DF}. Theorem~\ref{T1} gives a congruence between $E_{k,p}^{(-1)}$ and 
a newform $f\in S_k^{(-1)}(p)$. 
The newforms in $S_k^{(-1)}(p)$ are Galois conjugate over a field
of degree 41, and we verified the congruence by showing that the norms of the 
difference between Fourier coefficients are divisible by $\ell$. We also checked 
that the 
subspace of $W_{k-2}^+(p)^{\new}_{/ \FF_\ell}$ on which the Atkin-Lehner involution
acts as $-1$ has expected dimension 41, confirming the surjectivity in 
Theorem~\ref{T2}. 
\end{example}
}
\begin{example}\label{ex2}
One may ask whether the congruence in Theorem~\ref{T1}
comes from a congruence between the period polynomials 
$P^+(f)$ and $P^+(E_{k,p}^{(\e)})$ of the forms in the theorem (both normalized to
have constant term 1 at the identity coset). Such a congruence would imply the 
congruence of Hecke eigenvalues, and for level 1 it was shown to hold by 
Manin~\cite{Ma} in the cases when $S_k(1)$ is one dimensional. 
In higher level, this congruence often holds (see the next example), 
but the following example shows that it can also fail. For $p=5$, $k=40$, 
$\ell=71$, $\e=-1$, we have 
$\ell|p^{k/2}+\e$ and the congruence~\eqref{e1} holds for some $f\in S_k^{(\e)}(p)$, 
but $P^+(f)(I)\not\equiv P^+\big(E_{k,p}^{(\e)}\big)(I) \pmod \I$. This 
illustrates the fact that the Deligne-Serre lifting lemma guarantees the 
lift of systems of eigenvalues, but not of eigenvectors. 
\end{example}
\begin{example}\label{ex3} Let $M=7$, $k=6$, and let $g\in S_k(M)$ be the 
newform  
\[ g=  q - 10 q^2 - 14 q^3 + 68 q^4 - 56 q^5 + 140 q^6 - 49 q^7 +\ldots 
+1824q^{23}+\ldots\,. \] 
The cosets $A=\G_0(M)\sm **xy$ can be identified with points $(x:y)\in 
\PP^1(\Z/M\Z)$, and setting $P=P^+(g)$, we can use the relations 
$P|\dd=P$, $P|S=-P$ to express all 8 components of $P$ in terms of 
\[\begin{aligned}
  P((0:1))=-49X^4 + 1, \quad & 
  P((1:2))=80X^4 - \frac{43}2 X^3 - \frac{129}2 X^2 - 86X + 6, \\
  P((1:1))=-49X^4 + 49, \quad & 
  P((1:3))=-6X^4 - 86X^3 + \frac{129}2 X^2 - \frac{43}2 X - 80.
   \end{aligned}
\]
Since 43 divides all the coefficients except for those of $X^0$ and $X^4$, this 
illustrates the previous comment, namely we have 
$P^+(g)\equiv P^+\big(E_{k,M}^{(+1)}\big)\!\!\pmod{43}$.
This implies directly that $g\equiv E_{k,M}^{(+1)} \pmod{43}$, and  
indeed we have $43|7^3+1$, so the congruence follows from Theorem~\ref{T1}. 

To apply Theorem~\ref{T3}, we note that $\den P^+(g)=2$ in~\eqref{1.5}, 
so this condition poses no restriction. 
Taking $p=2$, we find $11|\lambda_p-p^{k/2-1}(p+1)$, and since 
$11\nmid p^{k/2-1}-1=3$ we deduce from Theorem~\ref{T3} that there exists 
a congruence between $g_p^{(-1)}$  
and a Hecke eigenform $f\in S_6(14)$ which is 2-new and has eigenvalue $-1$ for $W_2$. 
In fact, we find a newform~$f$ of level 14
(as predicted by the conjecture in the introduction)  with $q$-expansion 
\[f=q + 4q^2 + 8q^3 + 16q^4 + 10q^5 + 32q^6 - 49q^7+\ldots +2000q^{23}+\ldots \,. \]
\end{example}

\end{document}